\documentclass[11pt,twoside,reqno]{amsart}
\usepackage{times}
\usepackage{amsmath}
\usepackage{amssymb}
\usepackage{amsfonts}
\usepackage{amscd}
\usepackage{amsthm}
\usepackage{amsxtra}
\usepackage{stmaryrd}

\usepackage{mathrsfs}
\usepackage{nicefrac}
\usepackage{bbold}
\usepackage{graphicx}
\usepackage{color}
\usepackage{wrapfig}

\title[The VPN Tree Routing Conjecture for Outerplanar Networks]{The VPN Tree Routing Conjecture for Outerplanar Networks%
\\{\normalsize (Extended abstract)}%
}
\author{Samuel Fiorini}
\author{Gianpaolo Oriolo}%
\author{Laura Sanit\`a}
\address{Gianpaolo Oriolo \& Laura Sanit\`a: Dipartimento di Ingegneria dell'Impresa, Universit\`a di Roma ``Tor Vergata'', Rome, Italy.}
\email{oriolo@disp.uniroma2.it}
\email{laurasanita@gmail.com}
\author{Dirk Oliver Theis}
\address{Samuel Fiorini \& Dirk Oliver Theis: %
  Service de G\'eometrie Combinatoire et Th\'eorie des Groupes, %
  Department of Mathematics, %
  Universit\'e Libre de Bruxelles, %
  Brussels, Belgium}%
\email{sfiorini@ulb.ac.be}
\email{theis@uni-heidelberg.de}%
\thanks{Dirk Oliver Theis supported by \textit{Communaut\'e fran\c caise de Belgique -- Actions de Recherche Concert\'ees.}}

\date{Fri Nov 16 14:48:14 CET 2007}

\newtheorem{theorem}{Theorem}
\newtheorem{prop}[theorem]{Proposition}

\newtheorem{lemma}[theorem]{Lemma}

\theoremstyle{remark}

\newtheorem*{remark*}{Remark}

\newcommand{\RR}{\mathbb R}

\begin{document}

\maketitle

\begin{abstract}
The VPN Tree Routing Conjecture is a conjecture about 
the Virtual Private Network Design problem. It states 
that the symmetric version of the problem always has 
an optimum solution which has a tree-like structure.
In recent work, Hurkens, Keijsper and Stougie
(\textit{Proc.\ IPCO XI,} 2005; \textit{SIAM J.\ Discrete Math.,} 2007)
have shown that the conjecture holds when the
network is a ring. A shorter proof of the VPN Conjecture 
for rings was found a few months ago by Grandoni, Kaibel, 
Oriolo and Skutella
(to appear in \textit{Oper. Res. Lett.,} 2008). 
In their paper, Grandoni 
\emph{et al.} introduce another conjecture, called the 
Pyramidal Routing Conjecture (or simply PR Conjecture), 
which implies the VPN Conjecture. Here we consider a 
strengthened version of the PR Conjecture. First we 
establish several general tools which can be applied in
arbitrary networks. Then we use them
to prove that outerplanar networks satisfy the PR Conjecture. 
\end{abstract}

\section{Introduction}

The symmetric Virtual Private Network Design problem (sVPND)
takes place in an undirected network. Inside the network there 
are $k$ distinguished vertices called {\em terminals}. The goal 
of the problem is to choose a collection of ${k \choose 2}$ paths, 
one between each pair of terminals, and capacity reservations on 
the edges covered by these paths, in such a way that any admissible 
traffic demand between the terminals can be routed through the paths
and that the cost of the reservation is minimum (a precise definition
is given in the next paragraph). So far the question of determining 
the complexity of sVPND has remained open, see Erlebach and R\"uegg
\cite{ER04} and Italiano, Leonardi and Oriolo \cite{ILO06}. The 
\emph{VPN Tree Routing Conjecture\/} (or shortly, \emph{VPN Conjecture\/}) 
states that sVPND always has an optimum solution whose paths determine 
a tree. As shown by Gupta, Kleinberg, Kumar, Rastogi and Yener 
\cite{GKKRY01}, if the VPN Conjecture is true then sVPND can be 
solved in polynomial time by a single all-pairs shortest paths 
computation. Essentially, the only general class of networks where 
the VPN Conjecture is known to hold is the class of ring networks
(that is, whose underlying graph is a cycle), a result due to 
Hurkens, Keijsper and Stougie \cite{HKS05,HKS07}. (Actually, 
Hurkens {\em et al.} prove the VPN Conjecture in other cases
too, e.g., when the network is complete graph of size four.)
A short proof of the VPN Conjecture for ring networks was very
recently found by Grandoni, Kaibel, Oriolo and Skutella 
\cite{GKOS08}. A consequence of our results is that the 
VPN Conjecture holds for all outerplanar networks.

Let $G = (V,E)$ be a undirected, finite, simple graph representing
the network. Each edge $e$ of $G$ has a given cost $c_e \in 
\mathbb{R}_+$. Let $W \subseteq V$ denote the set of all terminals,
thus $|W| = k$. For each terminal $u \in W$ we are given an upper
bound $b_u \in \mathbb{Z}_+ $ on the total amount of traffic that 
$u$ can send or receive. The traffic demands between terminals
are encoded in a {\em traffic matrix\/}, that is, a non-negative 
$k \times k$ real matrix $D = (d_{uv})$ with lines and columns 
indexed by the set of all terminals which is symmetric and has 
zeroes on the diagonal. We say that a traffic matrix $D$ is 
{\em admissible\/} if we have $\sum_{v \in W} d_{uv} \le b_u$ for 
all terminals $u \in W$. Now the {\em symmetric Virtual Private 
Network Design problem\/} ({\em sVPND\/}) is to choose a simple
$u$--$v$ path $P_{uv}$ for each pair of distinct terminals $u, v 
\in W$ together with {\em capacity reservations\/} $y_e \in \mathbb{R}_+$ 
for all edges $e \in E$ in such a way that every admissible demand 
matrix $D = (d_{uv})$ can be routed, that is, $y_e \ge \sum (d_{uv} : 
u, v \in W \textrm{ and } e \in P_{uv})$, and the cost of the capacity
reservation $c^T y = \sum_{e \in E} c_e \, y_e$ is minimum.

Grandoni \emph{et al.} \cite{GKOS08} introduced a new problem related
to the symmetric Virtual Private Network Design problem which they
call Pyramidal Routing problem (PR). In this paper,
we consider the following version of their problem. As before, the 
network is given as a (undirected, finite, simple) graph $G = (V,E)$.
We will always assume that $G$ is connected. 
This time $G$ has a special vertex $r$ called the {\em root}. Each 
vertex $v$ has a certain {\em demand\/} $b_v \in \mathbb{Z}_+$. A 
vertex $v$ with $b_v > 0$ is referred to as a {\em terminal}. We 
always assume that the root is a terminal, i.e., we have $b_r \ge 1$.
Let $k$ be the integer defined as
$$
k := \sum_{v \in V} b_v.
$$
So $k$ is simply the total demand. A {\em routing} is a 
collection $\mathcal{P}$ of simple paths (repetitions are allowed) 
such that (i) all paths in $\mathcal{P}$ start at vertex $r$; (ii) 
for each vertex $v$ exactly $b_v$ paths of $\mathcal{P}$ 
end in $v$. In particular, any routing $\mathcal{P}$ contains 
$b_r$ trivial paths starting and ending at the root. 
The collection of paths $\mathcal{P}$ determine two vectors in 
$\mathbb{Z}^E$: the {\em $n$-vector} $n = n(\mathcal{P})$ and 
the {\em $y$-vector} $y = y(\mathcal{P})$ which are respectively 
defined as
\begin{eqnarray*}
n_e &:=& |\{P : e \in P \in \mathcal{P}\}|, \quad \mathrm{and}\\
y_e &:=& \min \{n_e,k-n_e\}.
\end{eqnarray*}
In other words $n_e$ gives the number of paths of $\mathcal{P}$
containing edge $e$ and $y_e = p(n_e)$ where $p : x \mapsto 
\min \{x,k-x\}$ is the so-called {\em pyramidal function}. Let 
$c_e \in \mathbb{R}_+$ denote the cost of edge $e \in E$. The 
cost of a routing is the total cost of its $y$-vector, that is, 
$c^T y = \sum_{e \in E} c_e y_e$. The {\em Pyramidal Routing 
problem} ({\em PR\/}) is to find a routing whose cost is minimum. 
As mentioned above, our version of PR is slightly more general 
than the one of Grandoni \emph{et al.} \cite{GKOS08}. The original
version is obtained by restricting $b_v$ to be $0$ or $1$ for all
vertices $v$. This is not a severe restriction because a demand
$b > 1$ at some vertex $u$ can be simulated, for instance, by 
adding $b$ pendant edges $uv_1$, \ldots, $uv_b$ with cost zero 
to the graph and letting $b_u = 0$ and $b_{v_i} = 1$ for $i = 1, \ldots, b$.

\medskip

\noindent \textbf{Conjecture} {\em (PR Conjecture). For any 
instance $(G,r,b,c)$ of the Pyramidal Routing problem there always 
exists an optimum routing whose paths form a tree (that is, such that 
the support of the corresponding $n$-vector induces a tree).} 

\medskip

As shown by Grandoni \emph{et al.} \cite{GKOS08}, the PR Conjecture
implies the VPN Conjecture. Moreover, it follows easily from their 
results that the PR Conjecture restricted to the class of outerplanar 
graphs implies the VPN Conjecture restricted to the class of 
outerplanar graphs.

We conclude this introduction by an outline of this extended 
abstract. In Section \ref{sec-tools} we gather several results 
which are at the same time basic and essential. First we note 
that if the PR Conjecture holds for all the blocks of a graph 
$G$ then it also holds for $G$. To a given graph $G$, root $r$ 
and demand vector $b$ we can associate in the standard way an 
upper-monotone polyhedron which we call the Pyramidal Routing
polyhedron (or PR polyhedron). 
%Solving a PR instance $(G,r,b,c)$ amounts to optimizing the 
%linear function $y \mapsto c^T y$ on the polyhedron. 
The PR Conjecture is equivalent to the following
statement: all extreme points of the PR polyhedron correspond to 
tree routings or, in other words, the $y$-vector of any routing
is dominated by (i.e., coordinate-wise bigger or equal to) a
convex combination of $y$-vectors of tree routings. We thus 
obtain a formulation of the PR Conjecture that does not involve 
edge costs. Finally, we give a necessary condition for a routing
to determine an extreme point of the PR polyhedron. In particular, 
our necessary condition implies that in any such extremal routing
all paths from the root to a given vertex must coincide. 

In Section \ref{sec-minor-closed} we show that if the 
PR Conjecture is satisfied by a graph then it is satisfied 
by all its minors. The proof uses the block-decomposition 
result of Section \ref{sec-tools} and the PR Conjecture 
for cycles, which we have to reprove because we consider 
a strengthened version of the PR Conjecture. The minor-monotonicity 
of the PR Conjecture is allows to focus on restricted classes of 
graphs. For instance, we can restrict to graphs with maximum degree 
at most three. 

We prove our main result in Section \ref{sec-outerplanar}.
More specifically, we prove the PR Conjecture for ladders
(i.e., $2$-connected outerplanar graphs with maximum 
degree at most three), which implies the PR Conjecture for 
outerplanar graphs, which in turn implies the VPN Conjecture 
for outerplanar graphs. 
%The proof goes by induction on the 
%number of rungs in the ladder. In the base case, there are
%two rungs and the ladder is a cycle. When there is at least 
%three rungs we focus on the end of the ladder which does not 
%contain the root in its interior. We consider a routing whose 
%support is the whole ladder and then prove that its $y$-vector 
%is dominated by a convex combination of $y$-vectors of routings
%with a smaller support. These new routings are obtained by 
%locally modifying the original routing. By induction, the
%$y$-vectors of these new routings are all dominated by 
%convex combinations of $y$-vectors of tree routings. The
%result follows.

%Finally, we point out that due to length restrictions we had 
%to leave out some of our proofs. Whenever necessary, we have 
%replaced the complete proofs by proof sketches.

\section{Fundamental tools}

\label{sec-tools}

Our first lemma allows us to reduce the PR Conjecture to 
$2$-connected graphs. We do not include its easy proof here
but point out that it relies in an essential way on the fact
that, in our version of the Pyramidal Routing problem, the 
demand $b_r$ at the root can be arbitrary.

\begin{lemma}
\label{lem-2-connected}
If the PR Conjecture holds for all blocks (maximal connected
subgraphs without a cut-vertex) of a graph $G$ then it holds 
also for $G$. \qed
\end{lemma}

%\begin{proof}
%Suppose that $G$ is not $2$-connected and let $G_1$, \ldots, $G_\ell$ 
%denote its blocks. For each $1 \le i \le \ell$, let $r_i$ denote
%the vertex of $G_i$ separating $G_i$ from the root $r$. We define a 
%new instance of the pyramidal routing problem inside each block, in 
%the following way. Consider a block $G_i$. We pick $r_i$ as the root 
%for $G_i$. The demand vector $b_i$ for $G_i$ is defined as 
%$$
%b_{i,v} := \left\{
%\begin{array}{ll}
%b_v                     
%&\textrm{if } v \notin \{r_1, \ldots, r_\ell\},\\
%b_v + \sum (b_w : v \textrm{ does not separate } w \textrm{ from } r)   
%&\textrm{if } v = r_i,\\
%b_v + \sum (b_w : v \textrm{ separates } w \textrm{ from } r)
%&\textrm{if } v \in \{r_1, \ldots, r_\ell\} \setminus \{r_i\}.
%\end{array}
%\right.
%$$
%The cost vector $c_i$ for $G_i$ is the restriction of the cost vector
%for $G$ to $E(G_i)$. We know that for each instance $(G_i,r_i,b_i,c_i)$
%of PR there is a tree routing which is optimal. By patching together these 
%tree routings we obtain a global tree routing for $G$ which is optimal. 
%%In fact, by construction, the $y$-vector of the routing is the 
%%concatenation of the $\ell$ $y$-vectors of each block and the 
%%total cost is the sum of individual costs.  
%\end{proof}

The following lemma provides a way to state the PR Conjecture 
without referring to edge costs. Given a graph $G=(V,E)$, root 
$r \in V$ and demands $b \in \mathbb{Z}_+^{V}$ we define the
{\em Pyramidal Routing polyhedron\/} (or {\em PR polyhedron\/})
$Q = Q(G,r,b)$ as the dominant of the convex hull of the $y$-vectors 
of routings in $G$. Thus we have
$
Q := \mathrm{conv} \{y(\mathcal{P}) \in \mathbb{R}^E : \mathcal{P} 
\textrm{ is a routing in } G\} + \mathbb{R}^E_+.
$
Solving a PR instance of the form $(G,r,b,c)$ amounts to minimizing
the linear function $y \mapsto c^Ty$ over the corresponding PR 
polyhedron $Q(G,r,b)$.
 
\begin{lemma}
\label{lem-cvx-PR}
The PR Conjecture holds for a certain graph $G$ if and only if all 
extreme points of the PR polyhedron correspond to tree routings. In 
other words, the PR Conjecture holds for $G$ if and only if 
for any routing $\mathcal{P}$ in $G$ there exists a collection of tree 
routings $\mathcal{T}_1$, \ldots, $\mathcal{T}_\ell$ and non-negative
coefficients $\lambda_1$, \ldots, $\lambda_\ell$ summing up to $1$ 
such that 
\begin{equation}
\label{eq-domin}
\sum_{i = 1}^\ell \lambda_i \, y(\mathcal{T}_i) \le y(\mathcal{P}).
\end{equation}
\end{lemma}

\begin{proof}
For the backward implication, let $\mathcal{P}$ be an optimum 
solution to PR with respect to some cost vector $c \in \mathbb{R}^E_+$. 
Then (\ref{eq-domin}) implies $\sum_{i = 1}^\ell \lambda_i \, c^T y(\mathcal{T}_i) \le c^T y(\mathcal{P})$. So at least one of the tree 
routings $\mathcal{T}_1$, \ldots, $\mathcal{T}_\ell$ has a cost which 
does not exceed the cost of $\mathcal{P}$. That is, at least one of the 
tree routings is optimum.

Let us now prove the forward implication by contradiction. Suppose that 
the PR Conjecture holds for $G$ but the PR polyhedron has an extreme
point $y(\mathcal{P})$ where $\mathcal{P}$ is not a tree routing. Then we 
can separate $y(\mathcal{P})$ from the other points of the PR 
polyhedron by a hyperplane. Because dominants are upper-monotone, 
it follows that there exists a non-negative cost vector $c$ such that 
$c^T y(\mathcal{P}) < c^T y(\mathcal{Q})$ for all routings $\mathcal{Q}$ 
such that $y(\mathcal{Q}) \neq y(\mathcal{P})$. In particular, we have
$c^T y(\mathcal{P}) < c^T y(\mathcal{T})$ for all tree routings 
$\mathcal{T}$, a contradiction. The result follows.
\end{proof}

If $y$ and $y'$ are two vectors in $\mathbb{R}^E_+$ such 
that $y' \le y$ we say that $y$ is {\em dominated\/} by
$y'$. So if a routing $\mathcal{P}$ satisfies (\ref{eq-domin}) for 
some choice of tree routings $\mathcal{T}_i$ and non-negative 
coefficients $\lambda_i$ summing up to $1$, then the $y$-vector of
$\mathcal{P}$ is dominated by the corresponding convex combination 
of $y$-vectors of tree routings. We call a routing {\em extremal\/} 
if its $y$-vector is an extreme point of the PR polyhedron. So 
proving the PR Conjecture amounts to proving that all extremal
routings are tree routings or, equivalently, that the $y$-vector
of any routing is dominated by a convex combination of $y$-vectors
of tree routings.

The next lemma provides a useful necessary condition for 
a routing to be extremal. It essentially says that if a 
routing $\mathcal{P}$ is extremal then its $n$-vector has 
to be an extreme point of the polytope defined as the convex 
hull of the $n$-vectors of all routings in $G$.

\begin{lemma}
\label{lem-n-domin}
Let $\mathcal{P}$ and $\mathcal{P}_1, \ldots, \mathcal{P}_\ell$ with $n$-vector 
$n(\mathcal{P})$ and $n(\mathcal{P}_1), \ldots, n(\mathcal{P}_\ell)$
respectively. Then
\begin{equation*}
  n(\mathcal{P}) = \sum_{i = 1}^\ell \lambda_i\,n(\mathcal{P}_i)
  \quad\text{implies}\quad
  \sum_{i = 1}^\ell \lambda_i\,y(\mathcal{P}_i) \le y(\mathcal{P}).
\end{equation*}
In particular, $\mathcal{P}$ is not extremal whenever 
some routing $\mathcal{P}_i$ whose corresponding coefficient 
$\lambda_i$ is positive has a $y$-vector distinct from that 
of $\mathcal{P}$.
\end{lemma}

\begin{proof}
This follows from the concavity of the pyramidal function 
$p : x \mapsto \min \{x,k-x\}$. Indeed, for each edge $e$
we have
\begin{equation*}
\sum_{i=1}^\ell \lambda_i\, y_e(\mathcal{P}_i)
= \sum_{i=1}^\ell \lambda_i\, p(n_e(\mathcal{P}_i))
\le p \Big(\sum_{i=1}^\ell \lambda_i\, n_e(\mathcal{P}_i)\Big)
= p(n_e(\mathcal{P})) = y_e(\mathcal{P}).
\end{equation*}
\end{proof}

\begin{wrapfigure}{4}{0mm}
  \scalebox{0.5}{\input{taming.pstex_t}}
\end{wrapfigure}
We will use the above lemma in quite an intricate way in Section
\ref{sec-outerplanar} below.  However, it will more often be used in a
very simple way in trying to \textit{tame} the behavior of the paths
in an extremal routing.  For a path $P$ from the root $r$ to some
terminal $u$ and a vertex $v$ on $P$, denote by $P^{rv}$ the sub-path
from $r$ to $v$ and by $P^{v-}$ the sub-path from $v$ to the terminal
$u$.  The picture on the right illustrates these definitions in the
context of the following ``taming'' lemma.

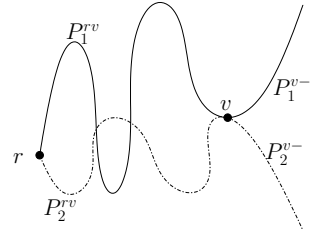
\begin{lemma}[Taming]\label{lem-2-paths}\label{lem:taming}
  Suppose $P_1$, $P_2$ are paths in $\mathcal{P}$ and $v$ is a vertex
  contained in both $P$ and $P_2$.  Assume that the vertex sets of
  $P_1^{rv}$ and $P_2^{v-}$ are disjoint, as well as those of
  $P_2^{rv}$ and $P_1^{v-}$.  Denote by $P_3$ the concatenation of
  $P_1^{rv}$ and $P_2^{v-}$ and by $P_4$ the concatenation of
  $P_2^{rv}$ and $P_1^{v-}$, we have
  \begin{equation*}
    \frac{1}{2} y(\mathcal{P}\setminus\{P_1\}\cup\{P_4\}) +
    \frac{1}{2} y(\mathcal{P}\setminus\{P_2\}\cup\{P_3\})
    \le y(\mathcal{P}).
  \end{equation*}
  In particular, in an extremal routing, all paths from the root to a
  fixed terminal coincide. \qed
\end{lemma}

%%%

\section{Minor-monotonicity of the conjecture}

\label{sec-minor-closed}

In this section we prove that the class of graphs for which 
the PR Conjecture holds is closed under edge deletions and 
contractions. This is a key ingredient in the proof of our
main result because it allows us to focus only on graphs 
with maximum degree at most $3$. Moreover, if the PR Conjecture 
turns out to be false then Proposition \ref{prop-minor-closed}
indicates that there could still be a hope to characterize the 
graphs which do satisfy the PR~Conjecture.

The next result states that the our \textit{more general}
PR~Conjecture is true for cycles.  It is a crucial ingredient in
proving the minor-monotonicity of the PR~Conjecture, and hence in the
out result on outerplanar graphs.

\begin{lemma}
\label{lem-cycle}
The PR Conjecture holds in case $G$ is a cycle.
\end{lemma}

We will sketch a proof below.

\begin{remark*}
  Since the version of the conjecture we consider here is more general
  than that of Grandoni \emph{et al.}  \cite{GKOS08},
  Lemma~\ref{lem-cycle} does not follow directly from their results.
  However, in a class of graphs which which is closed under
  edge-subdivisions, the two versions are equivalent.

  This can be shown as follows.  First, it is possible to show that an
  optimal tree routing for $0/1$-demands can be chosen to be a
  shortest path tree (this is proven similarly to the corresponding
  statement for the VPN, see \cite{GKKRY01}).  Second, one may use the
  same arguments as Grandoni \emph{et al.}~\cite{GKOS08} did for the
  VPN.
\end{remark*}

Grandoni \emph{et al.}  show that the $y$-vector of any routing on a
cycle is dominated by the $y$-vector of a single tree routing.  We now
sketch a proof of Lemma~\ref{lem-cycle}, which is a subtle
generalization of their argument: we dominate the $y$-vector by a
convex combination of at least two trees.  In fact, there are examples
showing that in our version of the conjecture, one tree may not be
sufficient.  Our approach generalizes to ladder graphs, as we will see
in Section~\ref{sec-outerplanar}.

\begin{proof}[Proof sketch]
We prove that any extremal routing in $G$ is a tree routing.
Without loss of generality we can assume that all vertices of the
cycle are terminals. If there is a non-terminal vertex $v$, 
then we can dissolve $v$, i.e., remove $v$ from the graph and make 
its two neighbors adjacent. (If the resulting graph is non-simple
then instead we use Lemma \ref{lem-2-paths} to conclude.)

Number the vertices and edges of the cycle consecutively as $w_0$, 
$e_0$, $w_1$, $e_1$, \ldots, $w_{m}$, $e_{m}$ where $w_0 =: r$ is 
the root.  Now let  $s_i := \sum_{j = 1}^i b_{w_i}$ 
for $i=0,\dots,m$ (i.e., $s_0=0$). 
Given any routing $\mathcal{Q}$, we define its \textit{$n$-function}
as the continuous function $f^{\mathcal Q} \colon [0,s_m] \to \mathbb{R}$
which satisfies $f^{\mathcal Q}(s_i) = n_{e_i}(\mathcal Q)$ for all
$i=0,\dots,m$ and is affine on every interval $[s_i,s_{i+1}]$. This 
is just the interpolation of the $n$-vector. We define the 
\textit{$y$-function} of $\mathcal{Q}$ as
$$
\bar f^{\mathcal Q} \colon [0,s_m] \to \mathbb{R} \colon 
t \mapsto \min \{f^{\mathcal Q}(t), k-f^{\mathcal Q}(t)\} 
= p(f^{\mathcal Q}(t)),
$$
so $\bar f^{\mathcal Q}(s_j) = y_{e_j}(\mathcal Q)$ for every index $j$.
Notice that the graph of $\bar f^{\mathcal Q}$ can be obtained from
that of $f^\mathcal{Q}$ by mirroring the part of the graph of 
$f^\mathcal{Q}$ above the {\em $k/2$-line\/}, that is, the horizontal 
line through the points $(0,k/2)$ and $(s_m,k/2)$.

Consider now an extremal routing $\mathcal{P}$ in $G$ and denote 
by $f$ its $n$-function and by $\bar f$ its $y$-function.  By 
Lemma~\ref{lem-2-paths} we can show that the affine segments of 
$f$ all have slopes in $\{-1,+1\}$.

By arguments similar to those we will use later in the 
proof of Theorem~\ref{th-outerplanar}, we can show that
the only interesting case is when the $n$-function $f$ 
consists of three affine parts: it increases on the 
interval $[0,s_j]$ for some $j$, then it decreases 
on the interval $[s_j,f(0)+s_j]$ and then increases 
again on the interval $[f(0)+s_j,s_m]$. Moreover, the
graph of $f$ crosses the $k/2$-line in the second 
interval, between $t = s_j$ and $t = f(0)+s_j$. Note 
that $f(0)+s_j = s_{j'}$ for some $j'$. 

Let $\mathcal T^1$ be the tree routing which omits edge 
$e_j$, and let $\mathcal T^2$ be the tree routing which omits
edge $e_{j'}$. Let $f^1$ and $f^2$ be the corresponding 
$n$-functions, and $\bar f^1$ and $\bar f^2$ the corresponding
$y$-functions. We show that $\bar f$ is dominated by a certain
convex combination $\lambda^1 \bar f^1 + \lambda^2 \bar f^2$
of $\bar f^1$ and $\bar f^2$. The coefficients $\lambda^1$ and
$\lambda^2$ are defined as
$$
\lambda^1 := \frac{\alpha_1}{\alpha_1 + \beta_1}
\quad
\textrm{and}
\quad
\lambda^2 := \frac{\beta_1}{\alpha_1 + \beta_1},
$$
where $\alpha_1 = \alpha_2$, $\beta_1 = \beta_2$ are as 
indicated in Fig.~\ref{fig:paths}. For the $y$-vector, this implies 
that we have $\lambda^1 y(\mathcal{T}^1) + \lambda^2 y(\mathcal{T}^2)
\le y(\mathcal{P})$. So either $\mathcal{P}$ is not extremal,
which contradicts our assumption, or $\mathcal{P}$ is a tree
routing. 
\begin{figure}[ht]
\centering
\scalebox{.6}{\input{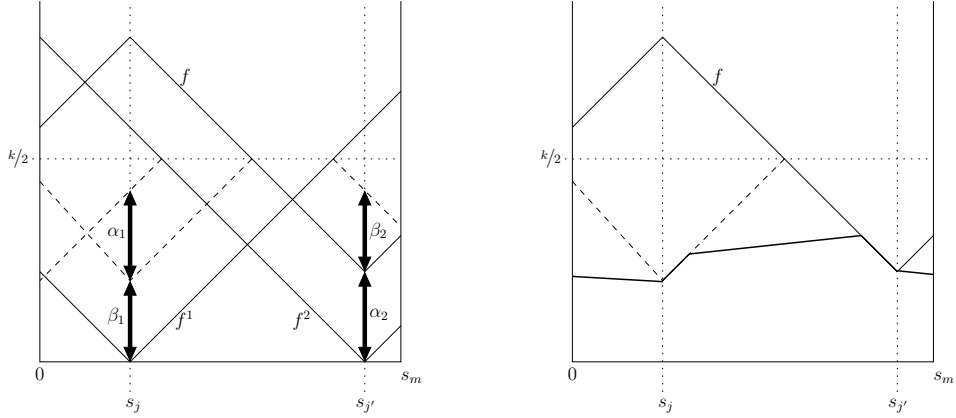}}
\caption{To the left: the graphs of the $n$-functions $f$, $f^1$ 
and $f^2$ (solid lines) and of the corresponding $y$-functions $\bar f$, 
$\bar f^1$ and $\bar f^2$ (the part of the graph of the $n$-function
which was reflected to obtain the graph of the $y$-function is
indicated by dashed lines). To the right: the graphs of $f$,
$\bar{f}$ and the convex combination $\lambda^1 \bar f^1 +
\lambda^2 f^2$ (in bold).}\label{fig:paths}
\end{figure}
\end{proof}

\begin{prop}
\label{prop-minor-closed}
If the PR Conjecture holds for $G$ then it holds for any 
minor of $G$.
\end{prop}

\begin{proof}
Graphs verifying the PR Conjecture are clearly closed under 
edge deletion. So it suffices to consider the case of edge
contractions. Let $e=st$ be an edge of $G$. Consider some 
instance $(G',r',b')$ of PR with $G' = G/e$ and some routing
$\mathcal{P'}$ in $G'$. Let $u_e$ denote the vertex of $G'$
resulting from the contraction of $e$. Now define a root $r$ 
and demands $b_v$ in $G$ as follows. If $r' \neq u_e$ then let 
$r := r'$, otherwise let $r := s$. Let $b_v := b'_v$ for $v \notin 
\{s,t\}$, $b_v := b_{u_e}$ for $v = s$ and $b_v := 0$ for $v = t$. 
Let $\mathcal{P}$ be the routing in $G$ obtained from $\mathcal{P}'$ 
by rerouting the paths containing $u_e$ in such a way that no path 
in $\mathcal{P}$ starts or ends in $t$ or uses an edge of $G$ incident 
to $t$ and to a common neighbor of $s$ and $t$. 

Consider $n := n(\mathcal{P})$, $n' := n(\mathcal{P}')$, 
$y := y(\mathcal{P})$ and $y' := y(\mathcal{P}')$. In order 
to relate $n'$ and $n$ (resp.\ $y'$ and $y$) we associate to 
each edge $f'$ in $G'$ a unique edge $f$ in $G$ as follows. We 
let $f := f'$ when $f'$ is not incident to $u_e$ or if $f = vu_e$ 
and $v$ is not a common neighbor of $s$ and $t$ in $G$; otherwise
$f' = wu_e$ for some common neighbor of $s$ and $t$ in $G$ and we
let $f := ws$. Then we have $n'_{f'} = n_{f}$ for all edges $f'$ 
of $G'$. In particular, we have $y'_{f'} = y_f$ for $f' \in E(G')$.

The PR Conjecture holds for $G$ so, by Lemma \ref{lem-cvx-PR},
the $y$-vector of $\mathcal{P}$ is dominated by a convex combination 
of $y$-vectors of tree routings, that is, $\sum_{i=1}^\ell \lambda_i\,y(\mathcal{T}_i) \le y(\mathcal{P})$ where the 
$\mathcal{T}_i$'s are some tree routings. Now let $\mathcal{T}'_i$ 
denote the routing obtained from the tree routing $\mathcal{T}_i$ by 
contracting the edge $e$. We remark that $\mathcal{T}'_i$ is not 
necessarily a tree routing. (It is if $\mathcal{T}_i$ uses the 
edge $e$.) Then we have $\sum_{i=1}^\ell \lambda_i 
y(\mathcal{T}'_i) \le y(\mathcal{P}')$. The support of each
$y(\mathcal{T}'_i)$ is either a tree or a tree plus an edge.
In the second case, by Lemmas \ref{lem-2-connected} and
\ref{lem-cycle}, $y(\mathcal{T}'_i)$ is dominated by a 
convex combination of $y$-vectors of tree routings. Hence
$y(\mathcal{P}')$ is dominated by a convex combination of 
$y$-vectors of tree routings. The result follows.
\end{proof}

From Proposition \ref{prop-minor-closed}, we can infer the
following. The PR-conjecture for graphs with maximum degree 
three implies the PR-conjecture. The PR-conjecture for (hexagonal) 
grids implies the PR-conjecture for planar graphs. The PR-conjecture 
for ladders implies the PR-conjecture for outerplanar graphs.
 
\section{Sketch of the proof for the outerplanar case}

\label{sec-outerplanar}

We come to the main result of this paper. 

\begin{theorem}
\label{th-outerplanar}
The PR-conjecture holds when $G$ is outerplanar.
\end{theorem}

The remainder of this section is dedicated to sketching the proof for
this result.

\begin{wrapfigure}{4}{0mm}
  \scalebox{.49}{\input{ladder.pstex_t}}
\end{wrapfigure}
By the previous remark, it suffices to consider the case
where $G$ is a ladder.  A \textit{ladder} is any graph obtained from a
matching $H$ with $E(H) = \{v_iv'_i : i = 1, \ldots, k\}$ by adding
$2k-2$ $H$-paths with no common internal vertex in such a way that
there is precisely one path from $v_i$ to $v_{i+1}$, and one path from
$v'_i$ to $v'_{i+1}$ for $i = 1, \ldots, k-1$; see the picture on the
right for an example (the letters refer to definitions below).  The
matching edges are called \textit{rungs} of the ladder.
%%

%
% We use induction on $\gamma = |E(G)| - |V(G)| 
% + 1$. For $\gamma = 1$, we use Lemma \ref{lem-cycle}.
% Now assume $\gamma \ge 2$. Consider some extremal routing 
% $\mathcal{P}$. Call $H$ the subgraph induced by the edges 
% which are used by $\mathcal{P}$. If $H$ is not $2$-connected 
% then by Lemma \ref{lem-2-connected}, $\mathcal{P}$ has to 
% be a tree routing (the nontrivial blocks are triangles). 
% Otherwise $H$ is still a sun so we can assume that $H = G$ 
% (if this is not the case the result follows by induction).
%

%%%%%%%%%%%%%%%%%%%%%%%%%%%%%%%%%%%
%%% DOT's macros:
\newcommand{\lt}{\left}
\newcommand{\rt}{\right}
%%%%%%%%%%%%%%%%%%%%%%%%%%%%%%%%%%%
%% \section{Sketch of the proof of the outerplanar case}
%
% We will now sketch the proof for the fact that every ladder graph satisfies the PR-conjecture.

Our proof of the PR~conjecture for ladders proceeds by induction on the number of rungs.  The case when $G$ has two rungs is done in Lemma~\ref{lem-cycle}.
Suppose now that a ladder $G$ with at least three rungs and demands $b_v$, $v\in V$, are given, and let $\mathcal P$ be any routing on $G$.

If there is an edge of $G$ which is not used by $\mathcal P$, then, possibly by using the block-decomposition Lemma~\ref{lem-2-connected}, we can invoke the
induction hypotheses to write the $y$-vector $y(\mathcal P)$ of this routing as a convex combination of $y$-vectors of tree routings as in \eqref{eq-domin},
Lemma~\ref{lem-cvx-PR}.
Otherwise, we will accomplish the induction step by producing a number of routings, $\mathcal P_1,\dots,\mathcal P_n$, each of which omits some edge of $G$, and
establish that $y(\mathcal P)$ is dominated by the convex hull of the $y$-vectors of these routings.  The induction hypothesis then yields, for each $\mathcal
P_j$, a convex combination of $y$-vectors of tree routings dominating $y(\mathcal P_j)$.  Thus, a convex combination of tree routings dominating $y(\mathcal P)$ is
found.

The strategy which we use is to focus on the lowest cycle of the ladder.  There we are able to ``uncross'' $\mathcal P$.  This strategy relies crucially on the
fact that the degrees of the vertices $u$, $v$ on the \textit{top edge} of the lowest cycle are not larger than three.  In fact, there are examples which show that
the strategy fails in general outerplanar graphs.  Thus, the minor-monotonicity is indispensable for our proof.

The strategy is implemented in the following steps.
\begin{enumerate}
\item Examine in what ways a path may meddle with the lowest cycle (cf. Section~\ref{ssec:examine-meddle}).
\item ``Smooth'' the routing to reduce the amount ``wobbling'' (cf. Section~\ref{ssec:smooth}).
\item For the smoothed routing, identify routings omitting edges and establish a convex combination of their $y$-vectors (cf. Section~\ref{ssec:cvxcomb}).
\end{enumerate}
The last step is the most onerous one.  The key ingredient there is $n$-functions and $y$-functions
similarly to those we use for the proof of Lemma~\ref{lem-cycle}.

\subsection{Examine in what ways a path may meddle with the lowest cycle}\label{ssec:examine-meddle}

Let $uv$ be the top edge of the lowest cycle, let $U$ be the vertex set of the lower connected component of $G-\{u,v\}$, and let $\bar U$ be the subgraph of $G$
induced by $U\cup\{u,v\}$ (see the above picture of a ladder).  W.l.o.g., we may assume that the root is not in $U$.  In what way may a path intersect $\bar U$?
\begin{itemize}
\item It may enter $\bar U$ and then head for a terminal which is not in $U$.  We call these paths \textit{thru paths,} and denote their number by $q$.
\item It may enter $\bar U$ and end at a terminal in $U$.  There are four ways how this can be done.  We symbolize them by strings $rXt$, where $X$ is replaced by
  the intersection of the path with $\{u,v\}$, taking into account the order in which the vertices $u$ and $v$ are visited on the path from the root to the
  terminal.  The following patterns are possible: $rut$, $rvut$; $rvt$, $ruvt$.
\end{itemize}

\begin{wrapfigure}{4}{0mm}%
  \begin{tabular}{|l|ll|}
    \hline
    ~        & $rut$ & $rvut$ \\
    \hline
    $rvt$    &  A    &  B   \\
    $ruvt$   &  B    &  A'\\
    \hline
  \end{tabular}%
\end{wrapfigure}
Note that, since $u$ and $v$ have degree three, if a path intersects $\{u,v\}$ in both $u$ and $v$, then it must contain the edge $uv$.  An application of the
Taming Lemma~\ref{lem-2-paths} yields that we can assume that $rut$ and $rvut$ do not both occur in $\mathcal P$: Such paths, starting from $r$, enter $U$ via $u$,
so their sub-paths connecting $r$ to $u$ can be interchanged.  The same holds for $rvt$ and $ruvt$.  We are faced with the matrix of cases depicted on the right.
Clearly, the two cases marked B are symmetric.  It is easy to see that they can be reduced to the cycle, i.e., to Lemma~\ref{lem-cycle}.  (We emphasize that we
need the full strength of Lemma~\ref{lem-cycle} here, i.e., demands $>1$ on the root of the cycle may occur.)

As for Case~A, we need to invoke the following easy lemma, which holds independently from our case distinction, and will be used for Case~A', too.

\begin{lemma}
  The $y$-vector of any routing $\mathcal P$ is dominated by a convex combination of tree routings and routings in which no thru path uses the top edge $uv$.
\end{lemma}
\begin{proof}[Proof (sketch).]
  This follows again by an application of the Taming Lemma~\ref{lem-2-paths}: we can dominate $y(\mathcal P)$ by routings in which \textsl{either} all thru paths
  use the top edge $uv$, \textsl{or} they all walk around $U$.  For the first type, we can remove $U$ from $G$, suitably update the demands for the vertices $u$
  and $v$, and invoke the induction hypotheses.
\end{proof}

\begin{wrapfigure}{4}{0mm}
  \scalebox{.75}{\input{U.pstex_t}}
\end{wrapfigure}
Hence, Case~A implies that the top edge $uv$ is not used by $\mathcal P$.  We are left with Case~A'.  The situation on $\bar U$ is visualized in the picture on the
right.  There are, say, $r_u\ge 0$, paths of type $ruvt$ entering $\bar U$ through the vertex $u$ and heading for a terminal in $U$ via $uv$, and, say, $r_v\ge 0$,
paths of type $rvut$ entering $\bar U$ through the vertex $v$ and heading for terminal in $U$ via $uv$.  In addition, there are $q$ thru paths entering $\bar U$ at
either $u$ or $v$, respectively, walking all the way along $U$, and either ending at $v$ or $u$, respectively, or leaving $\bar U$ again.  The numbers next to the
edges in the picture show known values of the $n$-vector for $\mathcal P$: The top edge $uv$ is used by $r_u+r_v$ paths, the topmost vertical edges by $r_v+q$ and
$r_u+q$ paths, respectively.

\subsection{Smoothing the routing}\label{ssec:smooth}

To study how $\mathcal P$ behaves on $\bar U\setminus\{uv\}$, we use $n$-functions and $y$-functions similarly to those in the proof of Lemma~\ref{lem-cycle}.
Here, we number the vertices and edges of $\bar U\setminus\{uv\}$ consecutively as $u=w_0$, $e_0$, \ldots, $w_{m}$, $e_{m}$.  Now let $s_i := \sum_{j=1}^i b_{w_i}$
for $i=0,\dots,m$.  For any routing $\mathcal Q$, we then define the $n$-function $f^\mathcal Q\colon[0,s_m]\to\RR$ and $y$-function $\bar f^{\mathcal
  Q}\colon[0,s_m]\to\RR$ precisely as in the proof of Lemma~\ref{lem-cycle}.  We abbreviate $f := f^\mathcal P$ and $\bar f := \bar f^\mathcal P$.

We say that a \textit{crossing} of any $n$-function $g\colon[0,s_m]\to\RR$ is a point $t$ in the open interval $\lt]0,s_m\rt[$ in which the function is smooth and
at which the graph of $g$ intersects the horizontal line through $(0,k/2)$ transversely, i.e., $g$ is affine near $t$, $g(t)=k/2$, and the slope of $g$ in $t$ is
$\pm 1$.
In general $f$ may have have many crossings, \textit{peaks} (i.e., local maxima in the open interval $\lt]0,s_m\rt[$) and \textit{valleys} (i.e., local minima in
the open interval): the function \textit{wobbles.}  The following lemma describes what can be done in such a situation.

\begin{lemma}
  Let $f^{\mathcal P}$ be the $n$-function of some routing $\mathcal P$.  Let $t_1$, $t_2$ be integers in $[0,s_m]$ with $f^{\mathcal P}(t_1) = f^{\mathcal
    P}(t_2)$.  By reflecting the graph of $f^{\mathcal P}$ in the interval $[t_1,t_2]$ at the horizontal line through $(0,f(t_1))$, we obtain the graph of an
  $n$-function of some routing.
\end{lemma}

% \begin{figure}[ht]
%   \centering
%   \scalebox{.5}{\input{unwobble.pdf_t}}
%   \caption{Smoothing: Reflecting to remove a valley and a crossing}
%   \label{fig:unwobble}
% \end{figure}

\begin{wrapfigure}{4}{0mm}
  \scalebox{.33}{\input{unwobble.pstex_t}}
\end{wrapfigure}
The figure on the right shows how this lemma is used to reduce the wobbling.  A valley where the function is above $k/2$ can be removed by reflecting as shown in
the picture.  Note that the $y$-function of the old routing is dominated by the $y$-function of the new routing.  Similarly, a peak below $k/2$ can be removed.
If $t_1$ and $t_2$ are two crossing, and there is no crossing between these two, then by reflection, we can obtain a routing whose $n$-function has fewer
crossings, and whose $y$-function dominates the $y$-function of the routing we started with.  In this manner, we can \textit{smooth} the $y$-function to arrive at
the following result.

\begin{lemma}
  After smoothing, we have an $n$-function which has at most one crossing, at most one peak, and at most one valley.  At the peak, the $n$-function must be
  above $k/2$, at the valley, it must be below $k/2$.
\end{lemma}

\subsection{Identify routings omitting edges and establish a convex combination of their $y$-vectors}\label{ssec:cvxcomb}

The following procedure is at the heart of our proof.  Similarly to what we did in our proof of Lemma~\ref{lem-cycle}, we establish path routings $\mathcal Q^1$
and $\mathcal Q^2$ omitting edges $g^1$ and $g^2$ of $\bar U$ respectively and identify coefficients $\lambda^1$ and $\lambda^2$ for the $y$-vectors $y(\mathcal
Q^1)$ and $y(\mathcal Q^2)$ such that $y(\mathcal P)$ is dominated by $\lambda^1 y(\mathcal Q^1) + \lambda^2 y(\mathcal Q^2)$.  This completes the proof.

We now describe how these routings, edges, and coefficients are identified in the case when $f$ has a crossing (the case without crossing is easier).  The proof
that these selections make sense and do the job is grossly beyond the limit of this extended abstract.  After smoothing the $y$-function $f=f^{\mathcal P}$ looks
as drawn in the left picture of Fig.~\ref{fig:sine-antisine}.  The dashed line is the part were $\bar f$ differs from $f$.  If $f$ has a peak, then we denote the
point by $s_j$ and if it has a valley, we denote it by~$s_{j'}$.

\begin{figure}[ht]
  \centering
  \scalebox{.62}{\input{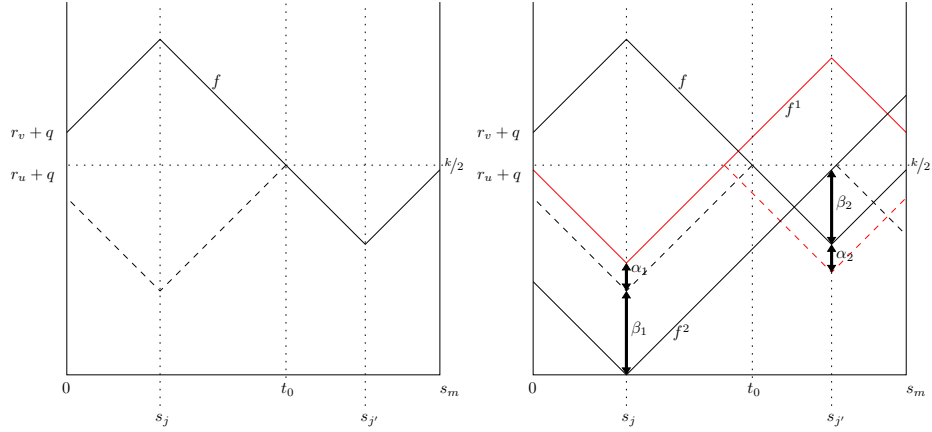}}
  \caption{Left: Smoothed solution. Right: Routings $\mathcal Q_i$ and coefficients}
  \label{fig:sine-antisine}
\end{figure}

A path ending at a terminal vertex might either arrive there in a clockwise or in an anti-clockwise motion.  We first construct a routing $\mathcal Q^1$ by
reversing this orientation for every terminal vertex in $U$.  $\mathcal Q^1$ clearly does not use the top edge $uv$, because we are in the Case~A', where every
path of $\mathcal P$ uses this edge.  The graph of the $n$-function $f^1:=f^{\mathcal Q^1}$ of $\mathcal Q^1$ can be obtained by reflecting the graph of $f$ at the
horizontal line through $(r_u+r_v)/2 + q$.  The right picture in Fig.~\ref{fig:sine-antisine} shows $f^1$ and $\bar f^1 := \bar f^{\mathcal Q^1}$.  Depending on
whether $k/2$ is above or below $(r_u+r_v)/2 + q$, $\bar f^1$ lies above $\bar f$ in $s_j$ and below $\bar f$ in $s_{j'}$ or vice versa.  Let us assume the former,
as in Fig.~\ref{fig:sine-antisine}.  For the other point $s_j$, we produce a routing which does not use the edge $e_j$, but coincides with $\mathcal P$ on every
edge not in $\mathcal P$.  Such a routing is uniquely determined except on the top edge $uv$, where for a certain number of paths, we may be able to choose whether
they use $uv$ or not (the respective values of $n_{uv}(\cdot)$ are easy to compute).  We take the routing $\mathcal Q^2$ which uses $uv$ as sparingly as possible.
Let $f^2:=f^{\mathcal Q^2}$, and $\bar f^2:=\bar f^{\mathcal Q^2}$.  Using the values $\alpha_i$ and $\beta_i$ as sketched in the figure, the coefficients for the
convex combination of the $y$-vectors are now
\begin{align*}
  \lambda^1 &:= \frac{\alpha_1}{\alpha_1 + \beta_1} & \lambda^2 &:= \frac{\beta_1}{\alpha_1 + \beta_1}.
\end{align*}
We then have $\bar f(s_j) = \lambda^1 \bar f^1(s_j) + \lambda^2 \bar f^2(s_j)$.  Clearly, we also have
\begin{equation}\label{eq:fs_j}
  \bar f(s_{j'}) = \frac{\alpha_2}{\alpha_2 + \beta_2} \,\bar f^1(s_{j'}) \;+\; \frac{\beta_2}{\alpha_2 + \beta_2} \,\bar f^2(s_{j'}),
\end{equation}
although this affine combination may not be a convex combination (moreover, there are cases when $\alpha_2+\beta_2\le 0$).  But it is possible to show that
\begin{align*}
  \alpha_2 &= \alpha_1 &\text{and}&& \beta_2 &\le \beta_1,
\end{align*}
and that the right hand side of equation~\eqref{eq:fs_j} is a decreasing function in $\beta_2$.  So we obtain $\bar f(s_{j'}) \ge \lambda^1 \bar f^1(s_{j'}) +
\lambda^2 \bar f^2(s_{j'})$.  The inequality in the other points in $[0,s_m]$ now follows using the fact that $\bar f$ is concave on the relevant intervals.

A computation shows that $y_{uv}(\mathcal P) \ge \lambda^1 y_{uv}(\mathcal Q^1) + \lambda^2 y_{uv}(\mathcal Q^2)$ also holds.

This completes our sketch of the proof of the Pyramidal Routing conjecture for ladders.

%%% Local Variables: 
%%% mode: latex
%%% TeX-master: "outerplanar.tex"
%%% fill-column: 163
%%% End: 

%%\footnotesize
\bibliographystyle{plain}
\bibliography{outerplanar}

\end{document}